\documentclass{amsart}

\usepackage{xcolor}

\newtheorem{theorem}{Theorem}[section]
\newtheorem{lemma}[theorem]{Lemma}
\newtheorem{corollary}[theorem]{Corollary}
\newtheorem{proposition}[theorem]{Proposition}

\theoremstyle{definition}
\newtheorem{definition}[theorem]{Definition}

\theoremstyle{remark}
\newtheorem{remark}[theorem]{Remark}

\numberwithin{equation}{section}

\newcommand{\bbN}{{\mathbb{N}}}
\newcommand{\bbR}{{\mathbb{R}}}

\newcommand{\bfJ}{{\mathbf{J}}}

\newcommand{\bbC}{{\mathbb{C}}}

\newcommand{\cB}{{\mathcal{B}}}

\newcommand{\bbJ}{{\mathbb{J}}}

\newcommand{\cR}{{\mathcal{R}}}

\newcommand{\calB}{{\mathcal{B}}}

\newcommand{\calU}{{\mathcal{U}}}

\newcommand{\ol}{\overline}
\newcommand{\linspan}{{\text{span}}}
\newcommand{\Tr}{{\text{Tr}}}

\newcommand{\vspan}{{\textrm{Span}}}
\newcommand{\ran}{{\textrm{ran}}}

\newcommand{\calR}{\mathcal{R}}
\newcommand{\calO}{\mathcal{O}}

\renewcommand{\Im}{\text{\rm Im\ }}

\allowdisplaybreaks \numberwithin{equation}{section}

\usepackage{enumerate}

\begin{document}
\date{\today}
\title[A Weyl matrix perspective]{A Weyl matrix perspective on Unbounded Non-Self-Adjoint Jacobi Matrices}

\author[B.\ Eichinger]{Benjamin Eichinger}

\address{School of Mathematical Science, Lancaster University, Lancaster LA1 4YF, United Kingdom}

\email{b.eichinger@lancaster.ac.uk}

\thanks{B.\ E.\ was supported by the Austrian Science Fund FWF, project no: P33885}

\author[M.\ Luki\'c]{Milivoje Luki\'c}

\address{Department of Mathematics, Rice University, Houston, TX~77005, USA}

\email{milivoje.lukic@rice.edu}

\thanks{M. L.\ was supported in part by NSF grant DMS--2154563.}

\author[G.\ Young]{Giorgio Young}
\address{Department of Mathematics,
  The University of Michigan,
 Ann Arbor, MI 48109, USA}
  \email[]{gfyoung@umich.edu}
  \thanks{G.Y.\  was supported by the NSF through grant DMS--2303363.}

\maketitle





\begin{abstract}
A new way of encoding a non-self-adjoint Jacobi matrix $J$ by a spectral measure of $|J|$  together with a phase function was described by Pushnitski--\v Stampach in the bounded case. We present another perspective on this correspondence, based on Weyl functions instead of moments, which simplifies some proofs and generalizes the correspondence to the unbounded case. In particular, we find a bijection between proper Jacobi matrices with positive off-diagonal elements, and a class of spectral data. We prove that this mapping is continuous in a suitable sense. To prove injectivity of the map, we prove a local Borg--Marchenko theorem for unbounded non-self-adjoint Jacobi matrices in this class that may be of independent interest. 
\end{abstract}

\section{Introduction}
Half-line Jacobi matrices 
\begin{align*}
J = \begin{pmatrix}
b_0& a_0&0&&&&\\
a_0& b_1&a_1&0&&&\\
0&a_1&b_2&a_2&0&&\\
&\ddots&\ddots&\ddots&\ddots&\ddots\\
\end{pmatrix}
\end{align*}
are most often studied in the formally self-adjoint setting $a_n > 0$ and $b_n \in \bbR$, and a central object in their spectral theory is the canonical spectral measure associated with the cyclic vector $\delta_0$. By a classical result known as Favard's theorem, this is a bijective correspondence between  bounded self-adjoint Jacobi matrices and compactly supported probability measures on $\bbR$.

Pushnitski--\v Stampach \cite{PushnitskiStampach} studied the more general setting of bounded \emph{non-self-adjoint} Jacobi matrices, with coefficients $a_n > 0$ and $b_n \in \bbC$, and described a bijective correspondence with a new kind of spectral data for $J$: a pair $(\nu,\psi)$ where $\nu$ is a spectral measure for $\lvert J \rvert$ and $\psi$ a phase function. More precisely, $\nu$ is a positive measure on $[0,+\infty)$ and $\psi \in L^\infty(\nu)$ a function with $\lvert \psi \rvert \le 1$ $\nu$-a.e.\ such that 
\begin{align}
\langle  \delta_0, f(|J|) \delta_0\rangle & =\int\limits_{[0,+\infty)} f \,d\nu,  \qquad \forall f \in \cB_b(\bbR), \label{eq:nudef} \\
\langle \delta_0,Jf(|J|)\delta_0\rangle & =\int\limits_{[0,+\infty)} sf(s)\psi(s)d\nu(s),  \qquad \forall f \in \cB_b(\bbR). \label{eq:phasedef}
\end{align}
We denote the set of bounded Borel functions on $\bbR$ by $\cB_b(\bbR)$ and use the convention that inner products are linear in the second parameter. Since $s\psi(s)$ appears in the integral \eqref{eq:phasedef}, the value $\psi(0)$ is arbitrary, and to ensure uniqueness, a value of $\psi(0)$ must be fixed. Since we will make use of the odd extension of $\psi$ to $\bbR$, we find it natural to set $\psi(0)=0$. The same authors studied this setting from the perspective of symmetric anti-linear operators in \cite{PushStamp}.

In this paper, we study unbounded, non-self-adjoint Jacobi matrices, with Jacobi parameters $a=(a_n)_{n\in \bbN_0}\in (0,\infty)^{\bbN_0}$ and $b=(b_n)_{n\in\bbN_0}\in \bbC^{\bbN_0}$. We emphasize that the diagonal parameters $b_n$ are complex valued, and that the coefficient sequences are not assumed to be bounded. Accordingly, $J$ is an unbounded operator on $\ell^2(\bbN_0)$, which we take to act on the maximal domain
\[
D(J) =  \left\{ u \in \ell^2(\bbN_0) \mid \sum_{n=1}^\infty |a_{n-1}u_{n-1}+b_nu_n+a_nu_{n+1}|^2<\infty\right \}.
\]
This is commonly called the maximal operator. We will also denote it by $J(a,b)$ when we need to emphasize the Jacobi parameters $a,b$.  Let us denote by $\ell^2_c(\bbN_0)$ the set of compactly supported sequences in $\ell^2(\bbN_0)$ and denote by $J_0 = J_0(a,b)$ the restriction of the operator $J = J(a,b)$ to the domain $D(J_0) = \ell^2_c(\bbN_0)$. Its closure $\ol{J_0}$ is commonly called the minimal operator.

In \cite{Beckermann}, Beckermann introduced the class of proper complex Jacobi matrices with coefficients $a_n,b_n\in \bbC, a_n\ne 0$, and studied their connections to formal othogonal polynomials: a matrix is said to be proper if $J = \ol{J_0}$. 
 We also note that in the formally self-adjoint setting $a_n > 0$, $b_n\in \mathbb R$, a Jacobi matrix is proper if and only if it is limit-point, making operators in this class a natural point of departure from the perspective of direct spectral theory, as well as moment problems, since in the self-adjoint setting being limit point corresponds to determinacy of the associated moment problem, cf. \cite{simon2010szegHo}. We also note that Beckermann and Smirnova \cite{Beckermann_Smirnova_2004} shows that a complex Jacobi matrix is proper if and only if it corresponds to a determinate moment problem in a suitably generalized sense. While this notion of determinacy will not be relevant to us here, it may serve to further motivate the study of proper Jacobi matrices as a point of departure in the unbounded setting. Furthermore, in Section 2, we find that a sufficient condition for being proper is that $a\in\ell^\infty(\bbN_0)$; in particular, this condition is satisfied for discrete Schr\"odinger operators with complex, unbounded potentials.

The operator $J$ is closed and densely defined, so by von Neumann's theorem, $J^*J$ is self-adjoint on the natural domain $\{u \in D(J) \mid Ju \in D(J^*) \}$. We denote by
\[
\lvert J \rvert = \sqrt{ J^* J}
\]
its (positive) square root, which is self-adjoint; in particular, the Borel functional calculus in \eqref{eq:nudef}, \eqref{eq:phasedef} is still well-defined. The following theorem constructs the direct spectral map; it is a generalization of \cite[Theorem 2.2]{PushnitskiStampach}. The proof follows similarly, with some extra care taken regarding domains.

\begin{theorem}\label{thm:2.2}
Let $a=(a_n)_{n\in \bbN_0}\in (0,\infty)^{\bbN_0}$ and $b=(b_n)_{n\in\bbN_0}\in \bbC^{\bbN_0}$ and suppose $J(a,b)$ is proper. There is a unique $\nu$ on $[0,+\infty)$ such that \eqref{eq:nudef} holds. There exists $\psi\in L^\infty(\nu)$ with norm $\leq 1$ (as element of $L^\infty(\nu)$) such that \eqref{eq:phasedef} holds. After normalizing by $\psi(0)=0$, $\psi$ is unique. 
\end{theorem}

Given a pair $(\nu, \psi)$ as above, we denote by $\psi_{o}$ the odd extension of the function $\psi:[0,+\infty)\to \bbC$ 
to the real line, and $\nu_e$ the even extension of the half-line probability measure $\nu$ to the full-line: the unique measure which is even and whose pushforward under $x\mapsto \lvert x \rvert$ is $\nu$. Then we consider the matrix-valued measure $\mu$ given by
\begin{align}\label{eq:measureform}
d\mu = \begin{pmatrix}
1&\psi_{o} \\\psi_{o}^*&1
\end{pmatrix}d\nu_{e}
\end{align}
for $f^*$ denoting the complex conjugate of $f$, and similarly for elementwise conjugation of a sequence in $\bbC^{\bbN_0}$. We will always assume that $\mu$ is related to $(\nu,\psi)$ in this way; in particular, $\mu$ is always a matrix measure of the form \eqref{eq:measureform}, with the even/odd symmetries described here. We will refer to the pair $(\nu,\psi)$, or the matrix measure $\mu$, as the spectral data of $J$.

In order to state our main theorem, we will need to define a subclass of non-degenerate and determinate matrix measures. A matrix measure $\tilde \mu$ with all moments finite is called determinate if it is uniquely determined by its moments, and non-degenerate if any matrix polynomial $P(\cdot)\in \bbC^{2\times 2}$ vanishing under the matrix inner product defined by $\tilde \mu$ is trivial,
\begin{align}\label{eq:nondeg}
\int\limits_\bbR P(x)^*d\tilde \mu(x) P(x)=0\implies P(x)= 0\quad \forall x\in\bbR.
\end{align}
A sufficient condition for determinacy is exponential decay of tails, we offer a proof of this in Lemma~\ref{prop:determinacy}. More generally, determinacy of the trace measure is sufficient, but not necessary for determinacy of the matrix measure, cf. \cite[Corollary 3.7]{berg2008matrix} and the ensuing discussion. Non-degeneracy is equivalent to the existence of a sequence of sequence $(P_n)_{n\geq 0}$ of orthonormal polynomials \cite[Theorem 4.6]{berg2008matrix}. By \cite[Lemma 6.1]{PushnitskiStampach}, a matrix measure of the form \eqref{eq:measureform} for $\nu_e$ of infinite support is always non-degenerate.

From a non-degenerate measure with all moments finite, one may form a block Jacobi matrix by assembling the coefficients appearing in the three term recurrence satisfied by a sequence of orthonormal polynomials into a tri-diagonal block matrix $\bbJ$, which may be viewed as an operator on the domain of compactly supported $\ell^2(\bbN_0)$ sequences. Building on earlier results \cite{Krein49mm}, Krein's work \cite{Krein49} shows that $\bbJ$ constructed this way has at least one deficiency index $0$ if and only if $\mu$ is determinate, see \cite[Theorem 3]{krein2016infinite} for an english translation. We define completely determinate measures to be the subset of non-degenerate measures with all moments finite that produce $\bbJ$ with both deficiency indices $0$. We note that this notion is well-defined, since different $\bbJ$ matrices corresponding to different orthonormal sequences differ by conjugation by a diagonal unitary operator, leaving the deficiency indices unchanged. 

\begin{definition}
Let $\tilde \mu$ be a non-degenerate matrix measure with all moments finite. We say $\tilde \mu$ is completely determinate if the associated block Jacobi matrices have deficiency indices both equal to $0$. 
\end{definition}
\begin{remark}\mbox{\,}
\begin{enumerate}
\item Our choice of terminology is inspired by complete determinacy being both a strengthening of the criterion for determinacy following from Krein's work, making this class a subset of the set of determinate non-degenerate matrix measures, as well as the already established definition of ``complete indeterminacy," where both deficiency indices are maximized, cf. \cite{guardeno2001matrix}. 
\item As discussed below \cite[Theorem 2]{krein2016infinite}, reality of the coefficients of $\bbJ$ is sufficient to conclude equality of the deficiency indices. This is by Von Neumann's theorem \cite[Theorem X.3]{ReedSimon2}. Generally, by the same argument, determinacy combined with a conjugation symmetry of the associated $\bbJ$ is sufficient to ensure equality of the deficiency indices at $0$ and complete determinacy. 
\end{enumerate}
\end{remark}

We may now state our main theorem, proving that conditions \eqref{eq:nudef}, \eqref{eq:phasedef}, \eqref{eq:measureform} produce a bijection between proper Jacobi operators $J$ and completely determinate matrix measures $\mu$ of the form \eqref{eq:measureform}.

\begin{theorem}
\label{thm:inverse}
The correspondence between $J$ and its spectral data $\mu$ determined by \eqref{eq:nudef}, \eqref{eq:phasedef}, \eqref{eq:measureform} is a bijection between the set of proper Jacobi matrices $J(a,b)$, with all $a_n>0, b_n\in \bbC$ and the set of completely determinate matrix measures $\mu$ of the form \eqref{eq:measureform}.
\end{theorem}
\begin{remark}
As discussed above, measures of the form \eqref{eq:measureform} are non-degenerate provided the trace measure has infinite (in the sense of cardinality) support \cite[Lemma 6.1]{PushnitskiStampach}. 
\end{remark}

This extends a bijection of Pushnitski--\v Stampach \cite{PushnitskiStampach}, who considered bounded $J$ and compactly supported $\nu$. We follow the central idea of their paper, to embed $J$ into a self-adjoint block operator.  However, our proofs are different, as they rely on matrix Weyl functions and the theory of Herglotz functions instead of moments. This allows us to offer some additional intuition for the phase function $\psi$. From our perspective, $\psi$ appears as one of the off-diagonal entries of a matrix probability measure that arises as the canonical spectral measure for a block Jacobi matrix unitarily equivalent to the self-adjoint block matrix formed through $J$. Such measures may be written as $d\mu=W(\cdot) d\nu_e(\cdot)$ with $W(\cdot)\in \bbC^{2\times 2}$ having constant trace $\nu_e$ almost everywhere, with $W\geq 0$ forcing $|\psi|\leq 1$. We compute that the diagonals of $W$ are equal, meaning that only $\nu_e$ and $\psi$ are needed to parametrize $\mu$; and $\nu_e$ is precisely the even extension of $\nu$, the spectral measure for $|J|$ described above.

To define the fundamental object in our work, we will use the following pair of vectors:
\[
\mathbf{e}_0:=\delta_0,\;\mathbf{e}_1:=J^*\delta_0,
\]
which is a cyclic set for $J^* J$ (see Theorem~\ref{thm: 2.1}). We use the notation 
\begin{align*}
R_{A}(z)=(A-z)^{-1}
\end{align*}
for the resolvent of an operator $A$ at the point $z\in \bbC$. We encode $J$ by the $2\times 2$ matrix 
\begin{align}\label{eq:distortedM}
M(z)  = (M_{i,j}(z))_{i,j=0}^1, \quad  
M_{i,j}(z)=\langle \mathbf{e}_j,R_{J^*J}(z)\mathbf{e}_i\rangle.
\end{align}
This is a matrix Herglotz function; just as the Weyl $m$-function is a basic tool for self-adjoint Jacobi matrices, the Weyl matrix $M$ is a central object in this paper. We prove the following counterpart of a local Borg--Marchenko theorem for unbounded non-self-adjoint Jacobi matrices. This theorem will be used to prove the injectivity portion of Theorem~\ref{thm:inverse}. For $w\in \bbC\setminus [0,+\infty)$, we abbreviate the matrix exponential by
\begin{align*}
w^{\frac{\sigma_3}{4}}=\begin{pmatrix} w^{1/4} & 0 \\  0 & 1/w^{1/4} \end{pmatrix}, \quad \arg(w^{1/4})\in (0,\pi/2).
\end{align*}

\begin{theorem}\label{thm:LocalBorgMarchenko}
Let $J,\tilde J$, be unbounded proper Jacobi matrices with Jacobi parameters $(a_n)_{n\in\bbN_0},(\tilde a_n)_{n\in\bbN_0}\in (0,\infty)^{\bbN_0}$, $(b_n)_{n\in\bbN_0},(\tilde b_n)_{n\in\bbN_0} \in \bbC^{\bbN_0}$, and Weyl matrix functions $M,\tilde M$ defined by \eqref{eq:distortedM}. Suppose $N\geq 2$. Then the following are equivalent:
\begin{enumerate}
\item\label{item:fixedsim} For a fixed sequence $w_j$, and an $\epsilon>0$ so that $\arg(w_j)\in [\epsilon,2\pi-\epsilon]$, we have
\begin{align}\label{eq:asymptotic}
\left\| w_j^{\frac{\sigma_3}{4}}  \left( M(w_j)- \tilde M(w_j)  \right)w_j^{\frac{\sigma_3}{4}} \right\|=\mathcal{O}(w_j^{-N}),\; j\to \infty.
\end{align}
\item \label{item:coeffeq} $a_n= \tilde a_n$ and $b_n=\tilde b_n$ for all $n\leq N-2$.
\item\label{item:simall} For every $\epsilon>0$, and $w$ in the sector $\arg(w)\in [\epsilon,2\pi-\epsilon]$, we have
\begin{align}\label{eq:asymptotic1}
\left\|w^{\frac{\sigma_3}{4}}\left(M(w)- \tilde M(w)\right)w^{\frac{\sigma_3}{4}} \right\|=\mathcal{O}(w^{-N}), w\to \infty.
\end{align}

\end{enumerate}

\end{theorem} 
 
Simultaneously and independently, Pushnitski and \v Stampach have examined this question in the setting of Schr\"odinger operators \cite{PushStampSchroed}. Earlier results of this flavor in non-self-adjoint settings \cite{Weikard,Brown} used a different function (a scalar Weyl function that isn't Herglotz) and the results were conditional on asymptotic similarity along ``admissible" rays. Our  $M$ matrix  serves as a different kind of extension of the Weyl  function of the self-adjoint theory.

In \cite{PushnitskiStampach}, the authors also prove their bijection is a homeomorphism. In the unbounded setting, this statement is essentially more delicate on both on the direct and inverse side, as we will comment on further below. Broadly, the operators are only closed and there is not a natural notion of convergence for them, while on the inverse side, the measures are not compactly supported. However, we recover the following continuity statement for the direct spectral map.

\begin{theorem}\label{thm:continuity} 
The bijection between $J$ and its spectral data $(\nu,\psi)$ given by Theorem~\ref{thm:inverse} is continuous in the following sense. 
Let $a^N,a^\infty\in (0,\infty)^{\bbN_0}$ and $b^N,b^\infty\in \bbC^{\bbN_0} $, and suppose $J(a^N,b^N)$, $J(a^\infty,b^\infty)$ are proper. If $J(a^N,b^N)\to J(a^\infty,b^\infty)$ in the sense that 
for all $k\in \bbN_0$,
\begin{align}\label{eq:strong}
J(a^{N},b^{N})\delta_k\to J(a^\infty,b^\infty)\delta_k, \;N\to \infty,
\end{align}
then for all $h\in C_0(\bbR):=\{h\in C(\mathbb R):\lim_{x\to\pm\infty}h(x)=0\}$, 
\begin{align}\label{eq:weak}
\begin{split}
&\lim_{N\to\infty}\int\limits_{[0,+\infty)}h(x)d\nu_N(x)=\int\limits_{[0,+\infty)}h(x)d\nu(x),\\
&\lim_{N\to\infty}\int\limits_{[0,+\infty)}h(x)\psi_N(x)d\nu_N(x)=\int\limits_{[0,+\infty)}h(x)\psi(x)d\nu(x).
\end{split}
\end{align}

\end{theorem}

Finally we prove criteria characterizing self-adjointness of $J$, as well as the case $b_n\equiv 0$ in terms of the phase function $\psi$. These are analogs of parts ii) and iii) of \cite[Theorem 2.6]{PushStamp}. The proofs are of a similar flavor; the main differences are our proofs rely on resolvents and coefficient stripping instead of a combinatorial lemma. We note that the remaining parts of this theorem depend on a polar decomposition for bounded operators found in \cite{garcia2007complex}. 
\begin{theorem}\label{thm:specialclasses}
Let $a=(a_n)_{n\in \bbN_0}\in (0,\infty)^{\bbN_0}$ and $b=(b_n)_{n\in\bbN_0}\in \bbC^{\bbN_0}$ and suppose $J(a,b)=J$ is proper. Then:
\begin{enumerate}[(i)]
\item \label{itm:self-adjointness} $J$ is self-adjoint if and only if $\psi(s)\in \bbR$ for $\nu$-a.e. $s\geq 0$.
\item \label{itm:freecriterion} $b_n=0$ for all $n\geq 0$ if and only if $\psi(s)=0$ for $\nu$-a.e. $s\geq 0$.
\end{enumerate}
\end{theorem}

As noted above, the main inspiration for this work is \cite{PushnitskiStampach}, and we follow their idea of embedding $J$ into a simple self-adjoint operator that is unitarily equivalent to a block Jacobi operator. The direct and inverse spectral theory of these latter operators is well-studied and may be leveraged to prove direct and inverse spectral results for $J$. We note that in addition to extending this work to the unbounded setting, where there are the usual technical obstacles, the perspective taken below shortens many of the proofs of their main theorems in the bounded case, in particular avoiding some combinatorial lemmas in that work in favor of coefficient stripping. Furthermore, it provides a more general correspondence. 

The paper is organized as follows. In Section~\ref{section:Preliminary}, we state and prove some preliminary results on the operator theory for $J$, and state unbounded extensions of the two main technical inputs we use from \cite{PushnitskiStampach} that allow for the definition of the direct spectral map (Theorem~\ref{thm:2.2}). We also prove Theorem~\ref{thm:specialclasses} in this section. In Section~\ref{sectionBlockJacobi}, we consider block Jacobi matrices and study the interplay between a certain form of their coefficients and certain symmetries of their Weyl matrices.   In Section~\ref{section:injectivity}, we prove a simple connection between the matrix $M$ defined above, and the Weyl matrix $\mathcal R$ for a block Jacobi matrix related to $J$, and we prove Theorem~\ref{thm:LocalBorgMarchenko}, Theorem~\ref{thm:inverse}, Theorem~\ref{thm:continuity}, and \ref{thm:specialclasses}.

\subsection*{Acknowledgements} We are grateful to the anonymous referees, whose feedback greatly improved this manuscript.
 
\section{Preliminary results}\label{section:Preliminary}

 For $u,v\in \ell^2(\bbN_0)$ we define the Wronskian by
\begin{align*}
W_n(u,v)=a_{n}(u_{n+1}v_n-u_nv_{n+1}).
\end{align*}

\begin{lemma}\label{lem:101824}
For all $u,v\in D(J(a,b))$, the limit
\begin{align}\label{eq:Wronskian}
W_{+\infty}(u,v) = \lim_{n\to\infty}W_n(u,v)
\end{align}
 is convergent and
\begin{align}\label{eq:WronskiLimit}
W_{+\infty}(u,v)= \langle J(a,b^*)u^*,  v \rangle - \langle u^*,  J(a,b)v \rangle.
\end{align}
\end{lemma}

\begin{proof}
Let $u,v\in D(J(a,b))$. Then $u^*,v^*\in D(J(a,b^*))$, so by the Cauchy-Schwarz inequality, the sequence
\[
(J(a,b^*)u^*)^*_n v_n-u_n (J(a,b)v)_n
\]
is summable. By direct calculation, it is equal to $W_n(u,v) - W_{n-1}(u,v)$, with the convention $W_{-1} = 0$. By telescoping, the limit \eqref{eq:Wronskian} is convergent and \eqref{eq:WronskiLimit} holds.
\end{proof}

\begin{lemma} For any coefficient sequences with $a_n > 0$, $b_n \in \bbC$,
\begin{align}\label{eq:J0adjoint}
J_0(a,b)^*=J(a,b^*).
\end{align}
In particular, $J(a,b)^* = J(a,b^*)$ if and only if  $W_{\infty}(u,v)=0$ for all $u,v \in D(J(a,b))$.
\end{lemma}

\begin{proof}
Vectors $u,w\in \ell^2(\bbN_0)$ satisfy $\langle u,J_0 v\rangle =\langle w,v\rangle$ for all $v\in \ell^2_c(\bbN_0)$ if and only if $\langle u,J_0 \delta_k\rangle =\langle w,\delta_k\rangle$ for all $k\in\bbN_0$. This in turn is equivalent to $w_k=a_{k-1}u_{k-1}+b_k^*u_k+a_ku_{k+1}$ for $k\in\bbN_0$, taking the convention $u_{-1}=0$ for the sake of brevity. Since $w\in \ell^2$ by assumption, this is in turn equivalent to $u\in D(J(a,b^*))$ with $J(a,b^*)u=w$, verifying \eqref{eq:J0adjoint}.

Thus, $J_0(a,b) \subset J(a,b)$ implies $J(a,b)^* \subset J_0(a,b)^* = J(a,b^*)$. Then \eqref{eq:WronskiLimit} implies that
\[
D(J(a,b)^*) = \{ u^* \in D(J(a,b^*)) \mid W_\infty(u,v)=0 \; \forall v \in D(J(a,b)) \},
\]
and the characterization is proven.
\end{proof}

As an immediate Corollary, we may prove that for Jacobi matrices with $a_n>0$, Beckermann's \cite{Beckermann} definition of proper matrices is equivalent to $J(a,b)^* = J(a,b^*)$.

\begin{corollary}\label{cor:proper}
A Jacobi matrix $J(a,b)$ with $a_n>0$ and $b_n\in \bbC$ obeys $J(a,b)^* = J(a,b^*)$ if and only if 
\begin{align}\label{eq:proper}
\ol{J_0(a,b)}=J(a,b).
\end{align}
\end{corollary}
\begin{proof}
Both directions will use that since $\ol{J_0(a,b)}=(J_0(a,b)^*)^*$, we have after taking adjoints of both sides of \eqref{eq:J0adjoint},
\begin{align*}
\ol{J_0(a,b)}=J(a,b^*)^*.
\end{align*}
If $J$ is proper,
\begin{align*}
J(a,b^*)^*=J(a,b).
\end{align*}
since $J(a,b)$ is closed. Thus, \eqref{eq:proper} holds. Similarly, if \eqref{eq:proper} holds, then 
\begin{align*}
J(a,b)=\ol{J_0(a,b)}=J(a,b^*)^*\implies J(a,b)^*=J(a,b^*)
\end{align*}
taking adjoints and using that $J(a,b^*)$ is closed.
\end{proof}

The following condition shows, in particular, that discrete Schr\"odinger operators ($a_n \equiv 1$) are always proper:

\begin{corollary}
If $\sup_n a_n < \infty$, the operator $J(a,b)$ is proper.
\end{corollary}

\begin{proof}
Assume $\sup_n a_n < \infty$. For any $u,v\in \ell^2(\bbN)$, by the Cauchy--Schwarz inequality, the sequence $W_n(u,v)$ is in $\ell^1(\bbN)$ so $\lim_{n\to\infty} W_n(u,v) =0$.
\end{proof}

We will now prove unbounded analogs of certain statements from \cite{PushnitskiStampach}, starting with an unbounded version of \cite[Lemma 3.2]{PushnitskiStampach}:

\begin{lemma}
\label{lem:BorelJJstar}
Let $T$ be a densely defined closed operator. Then, for all bounded Borel functions $f\in \mathcal{B}_b(\bbR)$ and all $\psi\in D(T)$, $\varphi\in D(T^*)$, we have $f(\lvert T \rvert) \psi \in D(T)$, $f(\lvert T^* \rvert) \varphi \in D(T^*)$, and
\begin{align}\label{eq:BorelJJstar}
Tf(|T|)\psi =f(|T^*|)T\psi ,\qquad T^*f(|T^*|)\varphi=f(|T|)T^*\varphi.
\end{align}
\end{lemma}
\begin{proof}
Consider the set $M$ of functions $g\in \mathcal B_b(\bbR)$ such that for all $\psi \in D(T)$, we have $g(T^*T)\psi \in D(T)$ and
\begin{align}\label{eq:1116}
Tg(T^*T)\psi =g(TT^*)T\psi.
\end{align}
This set is a subalgebra of $\cB_b(\bbR)$ and it contains resolvents: if $z\in \bbC\setminus \bbR$ and $\psi \in D(T)$, then $\eta=R_{T^*T}(z)\psi \in D(T^*T) \subset D(T)$, so $\psi = (T^*T-z) \eta$ implies $T^*T\eta \in D(T)$,  $T \psi  = T T^* T \eta - z T \eta$, and finally $T \eta = R_{TT^*}(z) T \psi$.
Moreover, if $g_n\in M$ are uniformly bounded and converge pointwise to $g$, then $g \in M$, since $T$ is closed. Thus, $M = \cB_b(\bbR)$.

Setting $f(s)=g(s^2)$ gives the first half of \eqref{eq:BorelJJstar}, while the second half comes from exchanging $T$ with $T^*$ in the argument.
\end{proof}

With this in hand, we may now prove the unbounded analog of  \cite[Theorem 2.1]{PushnitskiStampach}: 

\begin{theorem} 
\label{thm: 2.1}
The set $\delta_0, J^* \delta_0$ is a cyclic set for $\lvert J \rvert$, i.e., 
\begin{align}\label{eq:cyclicity}
C_{|J|}(\delta_0,J^*\delta_0):=\overline{ \vspan \{ R_{|J|}(z)\delta_0,R_{|J|}(z)J^*\delta_0:z\in\bbC\setminus \bbR \}}=\ell^2(\mathbb N)
\end{align}
In particular, the multiplicity of the spectrum of $|J|$ is at most $2$. Moreover, $\delta_0$ is a maximal vector for $|J|$, i.e., its spectral measure is a maximal spectral measure for $\lvert J \rvert$.
\end{theorem}

\begin{proof}
Denote the closure of the span by $V$. 
By a density argument, cf. \cite[Lemma 8.44]{LukicBook} for every bounded Borel function $f$, $V$ contains $f(\lvert J \rvert) \delta_0$ and $f(\lvert J \rvert) J^* \delta_0$.

Now assume that $\psi \in D(\lvert J \rvert^{2n})$ and $c  \in \bbN$, and denote $f_c(x) = x^{2n}$ if $\lvert x \rvert \le c$ and $f_c(x) = 0$ otherwise. Then by dominated convergence, $f_c(\lvert J \rvert) \psi \to \lvert J \rvert^{2n} \psi$ as $c\to\infty$. In particular,  $\lvert J \rvert^{2n} \delta_0, \lvert J \rvert^{2n} J^* \delta_0 \in V$.

As noted in the proof of \cite[Lemma 3.1]{PushnitskiStampach}, each $\delta_j$ may be written as a finite linear combination of the elements $|J|^{2n}\delta_0$ and $|J|^{2m}J^*\delta_0$, for $n,m\geq 0$, so $\delta_j \in V$ for all $j$. Since $V$ is closed, $V = \ell^2(\bbN_0)$.

Finally, we show $\delta_0$ is maximal. Let $w$ be a maximal vector for $|J|$, which is guaranteed to exist by the spectral theorem. Let $B\subset [0,+\infty)$ be a Borel set, and suppose $\delta_0\in \ker(\chi_{B}(|J|))=\ran(\chi_B(|J|)^{\perp}$,  then in particular, for $z\in \bbC\setminus \bbR$,
\begin{align*}
0&=\langle \delta_0, \chi_B(|J|) R_{|J|}(z)w\rangle=\langle \delta_0, R_{|J|}(z)\chi_B(|J|) w\rangle \\
&=\langle R_{|J|}(z^*)\delta_0, \chi_B(|J|) w\rangle.
\end{align*}
Similarly, we have 
\begin{align*}
\langle R_{|J|}(z)J^*\delta_0,  \chi_B(|J|)w\rangle&=\langle R_{|J|}(z) \chi_B(|J|)J^*\delta_0, w\rangle\\
&=\langle R_{|J|}(z) J^*\chi_B(|J^*|)\delta_0, w\rangle
\end{align*}
using Lemma~\ref{lem:BorelJJstar}. So that 
\begin{align*}
\langle R_{|J|}(z)J^*\delta_0,  \chi_B(|J|)w\rangle&=\langle R_{|J|}(z) J^*\chi_B(|J^*|)\delta_0, w\rangle
=\langle R_{|J|}(z) J^*\overline{\chi_B(|J|)\delta_0}, w\rangle=0
\end{align*}
where we use that $J^*\delta_0=(J\delta_0)^*$ and the assumption. By \eqref{eq:cyclicity}, $\chi_B(|J|)w=0$, and by maximality, $\chi_{B}(|J|)\equiv 0$, so that $\delta_0$ is a maximal vector.
\end{proof}

Consider $\nu$, defined as above as the spectral measure for $|J|$ for the vector $\delta_0$, we have the following lemma.
\begin{lemma}\label{lem:momentsfinite}
The measure $\nu$ has  infinite support (in the sense of cardinality) and finite moments:
\begin{align}\label{eq:moments}
\int\limits_{[0,+\infty)}x^{k}d\nu(x)<\infty,\;k\in \bbN_0.
\end{align}
\end{lemma}	

\begin{proof}
If $\nu$ was supported on a finite set of points, there would exist a nontrivial polynomial $p$ such that $p=0$ $\nu$-a.e.. Supposing such a polynomial $p$ exists, and taking this polynomial to be even, we may write $p(x) = q(x^2)$, and we have
\[
\lVert q(\lvert J \rvert^2) \delta_0 \rVert^2 = \int \lvert q(x^2) \rvert^2 \,d\nu(x) = 0.
\]
Owing to the tridiagonal structure of $J$, we have 
\[
\lvert J \rvert^{2k} \delta_0 - \prod_{j=0}^{k-1} a_j^2  \; \delta_{2k} \in \vspan\{ \delta_j \mid 0 \le j \le 2k-1\},
\]
Due to this, if $q$ is a polynomial of degree $m$, then $\langle
\delta_{2m}, q(|J|^2) \delta_0 \rangle \neq 0$. In particular, for any
nontrivial polynomial $q$, $q(|J|^2) \delta_0 \neq 0$, leading to
contradiction.

Since $\lvert J \rvert^{2}$ preserves $\ell^2_c(\bbN_0)$, $\delta_0 \in D( \lvert J \rvert^{2k})$ for every $k\in \bbN_0$. In particular, by the spectral theorem, there exists $\psi \in \ell^2(\bbN)$ such that $\delta_0 = (\lvert J \rvert^2 - i)^{-k} \psi$. By the functional calculus, spectral measures of $\delta_0$ and $\psi$ for $\lvert J \rvert$ are related by
\[
d\nu = \frac 1{ \lvert x^2 - i \rvert^{2k}} \,d\mu_\psi
\]
so $\mu_\psi(\bbR) = \lVert\psi \rVert_2^2 < \infty$ implies $\int x^{4k} \,d\nu(x) < \infty$.
\end{proof}

We can now establish the direct spectral map, generalizing \cite[Theorem 2.2]{PushnitskiStampach}.

\begin{proof}[Proof of Theorem~\ref{thm:2.2}]
The measure $\nu$ is, by \eqref{eq:nudef}, the spectral measure of $\delta_0$ with respect to $\lvert J \rvert$. 
For $f\in \mathcal{B}_b(\bbR)$, we note that $f(|J|)=g(J^*J)$ for $g(s^{1/2})=f(s)$, $s\geq 0$, a bounded Borel function. So, $f(|J|)\delta_0\in D(J^*J)\subseteq D(J)$, and the left hand side of \eqref{eq:phasedef} makes sense for $f\in \calB_b(\bbR)$. We prove the estimate 
\begin{align}\label{eq:functional}
|\langle Jf(|J|)\delta_0,\delta_0\rangle|\leq \int\limits_{[0,+\infty)}s|f(s)|d\nu(s)
\end{align}
for a dense set of $f\in L^1(sd\nu(s))$, from which the result will follow by duality.  Let $f\in C_c(0,\infty)$. Then, factoring $f(s)=|f(s)|^{1/2}f(s)^{1/2}$ for 
\begin{align*}
f(s)^{1/2}:=\begin{cases}
f(s)/|f(s)|^{1/2},&f(s)\ne 0\\
0,&f(s)=0
\end{cases}
\end{align*}
and for $s>0$ set 
\begin{align*}
h(s)=s^{-1/2}|f(s)|^{1/2},\;g(s)=s^{1/2}f(s)^{1/2}
\end{align*}
with $h(0)=g(0)=0$. $h$ and $g$ are bounded and continuous on $[0,+\infty)$ and $f(s)=h(s)g(s)$. By Lemma~\ref{lem:BorelJJstar}, we have 
\begin{align*}
\langle Jf(|J|)\delta_0,\delta_0\rangle&=\langle Jg(|J|)h(|J|)\delta_0,\delta_0\rangle=\langle g(|J^*|)Jh(|J|)\delta_0,\delta_0\rangle
\end{align*}
since $h(|J|)\delta_0\in D(J^*J)\subseteq D(J)$, because $h(s)=\tilde h(s^2)$ for a $\tilde h\in \calB_b(\bbR)$ and all $s\geq 0$. Thus,
\begin{align*}
|\langle Jf(|J|)\delta_0,\delta_0\rangle|&=|\langle Jh(|J|)\delta_0,g(|J^*|)^*\delta_0\rangle|\leq \|Jh(|J|)\delta_0\| \|g(|J^*|)^*\delta_0\|
\end{align*}
 by the Cauchy-Schwarz inequality. We may conclude since 
 \begin{align*}
 \|Jh(|J|)\delta_0\|^2&=\langle h(|J|)\delta_0,J^*Jh(|J|)\delta_0\rangle=\int\limits_{[0,+\infty)}s^2|h(s)|^2d\nu(s)=\int\limits_{[0,+\infty)}s|f(s)|d\nu(s)
\end{align*}
and using the antiunitary equivalence of $|J|$ and $|J^*|$ under complex conjugation,
\begin{align*}
\|g(|J^*|)^*\delta_0\|^2&=\langle g(|J|)\delta_0,g(|J|)\delta_0 \rangle=\int\limits_{[0,+\infty)}|g(s)|^2d\nu(s)=\int\limits_{[0,+\infty)}s|f(s)|d\nu(s)
 \end{align*}
giving \eqref{eq:functional}. Thus, the linear functional $f\mapsto \langle Jf(|J|)\delta_0,\delta_0\rangle$ is bounded by $1$ on a dense subset of $L^1(sd\nu(s))$, and so extends to a linear functional with the same bound on $L^1(sd\nu(s))$, while $\calB_b(\bbR)\subseteq L^1(sd\nu(s))$ by Lemma~\ref{lem:momentsfinite}. In particular, there is a unique $\psi\in L^\infty(sd\nu(s))$ such that \eqref{eq:phasedef} holds. Setting the normalization $\psi(0)=0$ extends $\psi\in L^\infty(\nu)$ uniquely. 
\end{proof}

Finally, we prove the sufficiency condition mentioned in the introduction, making use of the well-known scalar condition of exponential tails, cf. \cite[Corollary 4.11]{schmudgen2017moment}.

\begin{lemma}\label{prop:determinacy}
For a matrix measure $d\mu$,  assume that 
\begin{align}\label{eq:1025}
\int\limits_{\bbR}e^{\epsilon |x|}d\Tr(\mu)(x)<\infty
\end{align} 
for some $\epsilon>0$.
Then, $d\mu$ is a determinate measure. In other words, if 
\begin{align}\label{eq:matrixmoments}
\int\limits_{\bbR}x^kd\mu=\int\limits_{\bbR}x^kd\tilde \mu,\quad \forall k\in \bbN_0
\end{align}
for a matrix measure $d\tilde\mu$ with finite moments, then $d\tilde \mu=d\mu$.
\end{lemma}

\begin{proof}
The Radon--Nikodym decomposition of a matrix measure with respect to its trace (see, e.g.,\cite[Lemma 6.37]{LukicBook}) gives representations 
\begin{align*}
d\tilde\mu(x)=\widetilde{W}(x)d\Tr(\tilde\mu)(x),\quad d\mu(x)=W(x)d\Tr(\mu)(x)
\end{align*}
for $\tilde W,W\geq 0$ and $\Tr(W)=1$ a.e. $\Tr(\tilde\mu)$, $\Tr(\mu)$ respectively. 
For ${v}\in \bbC^2$, we have by the Cauchy-Schwarz inequality:
\begin{align*}
\int\limits_{\bbR}e^{\epsilon |x|} v^* W(x) {v}d\Tr(\mu)(x)<\infty
\end{align*}
and the scalar measure $v^* W(x){v}d\Tr(\mu)(x)$ is determinate, so that \eqref{eq:matrixmoments} implies 
\begin{align*}
\int\limits_{\bbR}x^k v^* W(x) {v}d\Tr(\mu)(x)=\int\limits_{\bbR}x^k v^*  \widetilde W(x)v d\Tr(\tilde\mu)(x),
\end{align*}
so that for any ${v}\in \bbC^2$, $v^* W v d\Tr(\mu) = v^* \widetilde W v d\Tr(\tilde\mu)$. By the polarization identity, for any $u,v \in \bbC^2$, $u^* W v d\Tr(\mu) = u^* \widetilde W v d\Tr(\tilde\mu)$, so $\mu = \tilde \mu$. 
\end{proof}

\section{Block Jacobi matrices with a symmetry condition} \label{sectionBlockJacobi}

We consider block Jacobi matrices with $2\times 2$ blocks,
\begin{align}\label{eq:J0}
\mathbb J  = \begin{pmatrix}
B_0& A_0&0&&&&\\
A_0^*& B_1&A_1&0&&&\\
0&A_1^*&B_2&A_2&0&&\\
&\ddots&\ddots&\ddots&\ddots&\ddots\\
\end{pmatrix}
\end{align}
It is assumed that $\det A_j \neq 0$ and $B_j^* = B_j$. 
We will also assume that $\bbJ$ is self-adjoint with the maximal domain $D(\bbJ) \subset \ell^2(\bbN_0)$ given by the set of $u$ such that $\bbJ u \in \ell^2(\bbN_0)$.

\begin{definition}\label{def:calR}
The Weyl matrix function for $\bbJ$ is $\calR : \bbC \setminus \bbR \to \bbC^{2\times 2}$ defined by
\begin{align*}
\calR_{i,j}(z):&=\langle \delta_i,(\bbJ-z)^{-1}\delta_j \rangle, \qquad i,j \in \{0,1\}
\end{align*}
\end{definition}

We denote by $S_+$ the right shift operator $S_+ e_k= e_{k+1}$ on $\ell^2(\bbN_0)$ and consider the once stripped block Jacobi matrix $\mathbb J_1=(S_+^*)^2 \mathbb J S_+^2$ and its Weyl matrix $\calR_1$. The Schur complement formula implies the coefficient stripping formula for block Jacobi matrices,  cf. \cite[Proposition 2.14]{DamPushnSim08}:

\begin{lemma}[Matrix coefficient stripping]
The Weyl matrices of $\mathbb J$ and its once stripped matrix $\mathbb J_1$ are related by

\begin{align}\label{eq:oncestripped}
\calR_1(z)=-A_0^{-1}(\calR(z)^{-1}-B_0+z)(A_0^*)^{-1}
\end{align}	
\end{lemma}

Our goal is to relate the following particular form of the $2\times 2$ blocks,

\begin{align}\label{eq:blockform}
A_j=\begin{pmatrix}
0&a_j\\a_j&0
\end{pmatrix},\quad B_j=\begin{pmatrix}
0&b_j\\b_j^*&0
\end{pmatrix}
\end{align}
to a particular symmetry of the Weyl matrix.

By general principles, $\cR$ is a matrix Herglotz function, i.e., it obeys
\[
\frac 1{2i} ( \cR(z) -  \cR(z)^* ) > 0, \qquad z\in \bbC \setminus \bbR.
\]
Thus, $\cR$ has an integral representation involving a $2\times 2$ matrix measure on $\bbR$; we are interested in the case when this matrix measure has a particular symmetry:

\begin{definition}\label{def:Herglotz}
We say a matrix valued Herglotz function $F$ is symmetric if it may be written as a Stieltjes transform of the following form:
\begin{align*}
F(z)=\int\limits_{\bbR}\begin{pmatrix}
1&\psi_o(x)\\\psi_o^*(x)&1
\end{pmatrix}\frac{d\nu_e(x)}{x-z}
\end{align*}
for $\nu_e$ an even probability measure on $\bbR$, and $\psi_o$ an odd function satisfying $|\psi |\leq 1$ $\nu_e$ almost everywhere.
\end{definition}

\begin{lemma}\label{lem:blocks}
Let the Weyl-Matrix $\calR$ corresponding to a block Jacobi matrix be symmetric in the sense of Definition~\ref{def:Herglotz}. Then it has the nontangential asymptotics in $\bbC_+:=\{z\in\bbC:\Im(z)>0\}$:
\begin{align}\label{eq:Rexpansion}
\calR(z)&= - \frac{\mathbb I}z -\frac{B_0}{z^2}-\frac{A_0A_0^*+B_0^2}{z^3}+\mathcal{O}(z^{-4}), \quad z \to \infty,\; \arg z \in [\alpha,\pi- \alpha]
\end{align}
for any $\alpha > 0$, and the first diagonal/off-diagonal entries have the form:
\begin{align*}
B_0=\begin{pmatrix}
0&b_0\\ b_0^*&0
\end{pmatrix}, \quad A_0=\begin{pmatrix}
0&a_0\\a_0&0
\end{pmatrix}U_1
\end{align*}
for $U_1$ a unitary matrix and 
\begin{align}\label{eq:b0a0moments}
b_0=\int\limits_{\bbR} x\psi_{o}(x)d\nu_{e}(x), \quad a_0=\sqrt{\int\limits_{\bbR} x^2d\nu_{e}(x)-|b_0|^2}>0.
\end{align}
\end{lemma}

\begin{proof}
In a sector $\arg z \in [\alpha,\pi -\alpha]$,  $\alpha > 0$, symmetry of $\cR$ leads to nontangential asymptotics as $z\to \infty$, 
\begin{align}\label{eq:Rfullexpand}
\cR(z) = - \sum_{j=0}^{n-1} z^{-j-1} \int\limits_{\bbR} x^j \begin{pmatrix}
1&\psi_{o}(x)\\\psi_{o}(x)^*& 1
\end{pmatrix}d\nu_e(x) + O(z^{-n-1})
\end{align}
due to  the geometric expansion
\[
\frac {-1}{x-z} = \sum_{j=0}^{n-1} \frac{x^j}{z^{j+1}} + \frac{x^n}{z^n(x-z)}
\]
and the estimate $\lvert x-z \rvert \ge \Im z \ge  \sin \alpha \lvert z \rvert$.

In particular, since $\cR_1$ is of this form, $\calR_1(z)=-\frac1z \mathbb I+\calO(1/z^2),\;z\to\infty$ for $\arg(z)\in [\alpha,\pi- \alpha]$ for some $\alpha>0$. Using the coefficient stripping formula \eqref{eq:oncestripped} in the form $\calR(z)=(B_0-z-A_0\calR_1(z)A_0^*)^{-1}$ gives
\begin{align*}
\calR(z) & = -\frac1z\left(\mathbb I -\frac1z (B_0-A_0\calR_1(z)A_0^*)\right)^{-1}\\
&=-\frac1z\left(\mathbb I+\frac{B_0}{z}+\frac{A_0A_0^*+B_0^2}{z^2}+\mathcal{O}(1/z^3)\right).
\end{align*}
Equating coefficients, we see that 
\begin{align*}
B_0=\int\limits_{\bbR} x \begin{pmatrix}
1&\psi_{o}(x)\\\psi_{o}(x)^*& 1
\end{pmatrix}d\nu_e(x)=\begin{pmatrix}
0&b_0\\b_0^*&0
\end{pmatrix}
\end{align*}
where $b_0=\int\limits_{\bbR} x\psi_{o}(x)d\nu_{e}(x)$, since the measure is even. Similarly,  
\begin{align*}
A_0A_0^*&=\int\limits_{\bbR}x^2\begin{pmatrix}
1&\psi_{o}(x)\\\psi_{o}(x)^*& 1
\end{pmatrix}d\nu_e(x)-B_0^2=\begin{pmatrix}
a_0^2&0\\0&a_0^2
\end{pmatrix}
\end{align*}
for 
\begin{align*}
a_0^2:=\int\limits_{\bbR} x^2d\nu_e(x)-\left| \int\limits_{\bbR} x\psi_{o}(x)d\nu_{e}(x) \right|^2 \ge 0
\end{align*}
by the Cauchy-Schwarz inequality and since $\|\psi_{o}\|_\infty\leq 1$; we set $a_0$ to be the nonnegative square root of this quantity. 

Moreover, $\det A_0 \neq 0$ implies $a_0^2 \neq 0$. Thus, by the polar decomposition of $A_0$, there exists a unitary $U$ such that
\begin{align*}
A_0=\begin{pmatrix}
a_0&0\\0&a_0
\end{pmatrix}U=
\begin{pmatrix}
0&a_0\\a_0&0
\end{pmatrix}\begin{pmatrix} 0&1\\1&0\end{pmatrix}U
\end{align*}
and setting $U_1=( \begin{smallmatrix} 0&1\\1&0\end{smallmatrix}) U$ completes the proof.
\end{proof}

We record the following general facts about Stieltjes transforms of a complex measure. 
\begin{lemma}\label{lem:Cstieltjes} 
Let $\rho$ be a complex measure on $\bbR$ and let
\begin{align*}
m(z)=\int\limits_\mathbb R\frac{d\rho(x)}{x-z}, \quad z \in \bbC \setminus \bbR.
\end{align*}
Then, 

\begin{align}\label{eq:symmetry}
m(z)=\pm m(-z)\iff d\rho(x)=\mp d\rho(-x),
\end{align}
and 
\begin{align}\label{eq:reality}
m(z)^*=m(z^*)\iff d\rho(x)^*=d\rho(x).
\end{align}
where $d\rho(-x)$ denotes the pushforward of $\rho$ under the map $x\mapsto -x$. 
\end{lemma}
\begin{proof}
By Stieltjes inversion (see, e.g., \cite[Lemma 7.65]{LukicBook}), the Stieltjes transform uniquely determines the complex measure, so \eqref{eq:symmetry} follows from 
\begin{align*}
m(-z) = \int\limits_\bbR \frac{d\rho(x)}{x+z} = \int\limits_\bbR \frac{d\rho(-x)}{-x+z} = \int\limits_\bbR \frac{-d\rho(-x)}{x-z}. 
\end{align*}
Similarly, 
\begin{align*}
m(z)^* = \int\limits_\bbR \frac{d\rho(x)^*}{x-z^*} 
\end{align*}
so that by Stieltjes inversion again, we have \eqref{eq:reality}.
\end{proof}

We prove the following lemma, which, along with an inductive argument, will allow for the proof of Theorem~\ref{thm:LocalBorgMarchenko}:
\begin{proposition}\label{prop:stripping}
Suppose that the Weyl matrix $\calR$ corresponding to a block Jacobi matrix $\mathbb J$ is symmetric in the sense of Definition~\ref{def:Herglotz}. Suppose further that the first off-diagonal term $A_0$ is of the form \eqref{eq:blockform} for $a_0>0$.
Then the Weyl matrix $\calR_1$ corresponding to the once stripped block Jacobi matrix $\mathbb J_1=(S_+^*)^2 \mathbb J S_+^2$ satisfies 
\begin{align}\label{eq:strippedformula}
\calR_1(z)&=\frac{-1}{|a_0|^2}\begin{pmatrix}
\det(\calR(z))^{-1} \calR_{0,0}(z)+z&-\det(\calR(z))^{-1}\calR_{1,0}(z)-b_0^*\\-\det(\calR(z))^{-1}\calR_{0,1}(z)-b_0&\det(\calR(z))^{-1}\calR_{1,1}(z)+z
\end{pmatrix}.
\end{align}
In particular, $\calR_1$ is also symmetric.
\end{proposition}

\begin{proof}
Since $\mathcal R_1$ is a Weyl matrix for a block Jacobi matrix, it is the Stieltjes transform of a matrix measure $\mu$ satisfying $\mu(\bbR)_{i,j} = \langle \delta_j, \delta_i \rangle$ so $\mu(\bbR) = \mathbb I$. 

Decomposing $\mu$ with respect to the positive measure $\Tr(\mu)$, we may write $d\mu=W\, d\Tr(\mu)$ for $W:\bbR\to \bbC^2$, with $\Tr(W)=1$ and $W\geq 0$ $\Tr(\mu)$-a.e., so that 
\begin{align}\label{eq:StieltjesR1}
\mathcal R_1(z)=\int\limits_{\bbR}\frac{W(x)d\Tr(\mu)(x)}{x-z}
\end{align}
Using \eqref{eq:oncestripped} and denoting the entries of $\calR$ by $\calR_{i,j}$, we may use the form of $A_0$ and $B_0$ found in Lemma~\ref{lem:blocks} and denoting $\sigma_1=(\begin{smallmatrix}
0&1\\1&0
\end{smallmatrix})$, we may compute 
\begin{align*}
\calR_1(z)&=\frac{-\sigma_1}{a_0}\left(\det(\calR(z))^{-1}\begin{pmatrix}
\calR_{1,1}(z)&-\calR_{0,1}(z)\\-\calR_{1,0}(z)&\calR_{0,0}(z)
\end{pmatrix}-\begin{pmatrix}
0&b_0\\b_0^*&0
\end{pmatrix}+z\mathbb I \right) \frac{\sigma_1}{a_0^*}\\
&=\frac{-\sigma_1}{|a_0|^2}\begin{pmatrix}
\det(\calR(z))^{-1}\calR_{1,1}(z)+z&-\det(\calR(z))^{-1}\calR_{0,1}(z)-b_0\\-\det(\calR(z))^{-1}\calR_{1,0}(z)-b_0^*&\det(\calR(z))^{-1} \calR_{0,0}(z)+z
\end{pmatrix} \sigma_1\\
&=\frac{-1}{|a_0|^2}\begin{pmatrix}
\det(\calR(z))^{-1} \calR_{0,0}(z)+z&-\det(\calR(z))^{-1}\calR_{1,0}(z)-b_0^*\\-\det(\calR(z))^{-1}\calR_{0,1}(z)-b_0&\det(\calR(z))^{-1}\calR_{1,1}(z)+z
\end{pmatrix}
\end{align*}
yielding \eqref{eq:strippedformula}.

By \eqref{eq:symmetry} of Lemma~\ref{lem:Cstieltjes} and our assumption that $\calR$ is symmetric, the entries $\calR_{0,0}$ and $\calR_{1,1}$ are odd, and equal, while $\calR_{0,1}$ and $\calR_{1,0}$ are even.  Thus, the determinant of $\calR$ is even, and by our computation above, the diagonal entries of $\calR_{1}$ are odd and equal, while the off-diagonal entries are even. Thus, we may conclude by Lemma~\ref{lem:Cstieltjes} that the diagonal entries are Stieltjes transforms of an even complex measure $\nu_e$, while the off-diagonal entries are Stieltjes transforms of an odd complex measure. Comparing to \eqref{eq:StieltjesR1}, and Stieltjes inverting, we have
\begin{align*}
d\nu_e =W_{0,0} \, d\Tr(\mu) =W_{1,1} \, d\Tr(\mu),
\end{align*}
so that by the normalization $\Tr(W)=1$, we have $W_{0,0}=W_{1,1}=1/2$ $\Tr(\mu)$-a.e., and 
\begin{align*}
\nu_e= \tfrac 12 \Tr(\mu)
\end{align*} 
a probability measure since $\mu(\bbR) = \mathbb I$. Meanwhile, writing $W_{0,1}=W_{1,0}^*=\psi_o$, we see $|\psi_o |\leq 1$  $\nu_e$-a.e.\  since $W$ is positive definite $\Tr(\mu)$-a.e.
\end{proof}

\section{Coefficient stripping}\label{section:injectivity}

The central idea of \cite{PushnitskiStampach} is to embed the non-self-adjoint $J$ into a block matrix
\begin{equation}\label{eq:domain}
\bfJ= \begin{pmatrix}
0&J(a,b) \\
J(a,b^*)&0
\end{pmatrix}, \qquad D(\bfJ)=D(J(a,b^*))\oplus D(J(a,b)).  
\end{equation}

By the previous section, if $J = J(a,b)$ is proper, $\bfJ$ is self-adjoint on $\ell^2(\bbN_0) \oplus \ell^2(\bbN_0)$, and $\bfJ$ is the closure of the operator
\[
\mathbf J_0=\bfJ_0(a,b)=\begin{pmatrix}
0&J_0(a,b)\\ J_0(a,b^*)&0
\end{pmatrix}, \qquad D(\bfJ_0) = \ell^2_c(\bbN_0) \oplus \ell^2_c(\bbN_0).
\]
Moreover,  conjugation by the unitary map
\[
V:\ell^2(\bbN_0)\to \ell^2(\bbN_0)\oplus\ell^2(\bbN_0), \qquad 
V \delta_{2j}=\delta_j\oplus 0,\quad V \delta_{2j+1}=0\oplus \delta_j,
\]
turns this into $\mathbb J = V^* \bfJ V$ which has the form of a block Jacobi matrix \eqref{eq:J0}
with $2\times 2$ blocks of the form \eqref{eq:blockform}.
Denoting by  $\mathbb J_0$ the restriction of $\mathbb J$ to $\ell^2_c(\bbN_0)$, if $J$ is proper, then $\mathbb J$ is the closure of $\mathbb J_0$.

We derive the following relationships between the Weyl matrix $\calR$ for $\bbJ$, the quantities $M_{i,j}$ defined by \eqref{eq:distortedM}, and the spectral data $(\nu,\psi)$ of $J$:

\begin{proposition}\label{prop:RMrelationship}
For $z\in\bbC\setminus \bbR$, 
\begin{align}\label{eq:RM} 
\mathcal R(z)=\begin{pmatrix}
 z M_{0,0}(z^2) & M_{1,0}(z^2)\\ M_{0,1}(z^2)& z M_{0,0}(z^2)
\end{pmatrix}
\end{align}
and 
\begin{align}\label{eq:Rstructure}
\mathcal R(z)=\int\limits_{\bbR}\begin{pmatrix}
1&\psi_{o}(x)\\\psi_{o}(x)^*& 1
\end{pmatrix}\frac{d\nu_e(x)}{x-z}.
\end{align}

In particular, $(\nu,\psi)$ uniquely determine the matrix-valued Herglotz function $\mathcal{R}$. Moreover, the diagonal entries of $M$ are related by
\begin{equation}\label{eq:Mdiagonalentryidentity}
zM_{0,0}(z^2)=\frac1z(-1+M_{1,1}(z^2)).
\end{equation}
\end{proposition}

\begin{proof}
From \eqref{eq:nudef}, \eqref{eq:phasedef} we obtain the following integral representations for the entries of $M$:
\begin{align*}
M_{0,0}(z^2)&=\langle \mathbf e_0,(J^*J-z^2)^{-1}\mathbf e_0\rangle=\langle \delta_0,(J^*J-z^2)^{-1}\delta_0\rangle\\
& =\int\limits_{[0,+\infty)} \frac{d \nu(x)}{x^2-z^2},
\end{align*}
and
\begin{align*}
M_{0,1}(z^2)&=\langle \mathbf e_1,(J^*J-z^2)^{-1}\mathbf e_0\rangle=\langle \delta_0,J(J^*J-z^2)^{-1}\delta_0\rangle \\
&=\int\limits_{[0,+\infty)} \frac{x\psi(x)d \nu(x)}{x^2-z^2}.
\end{align*}
and
\begin{align*}
M_{1,0}(z^2)&=\langle \mathbf e_0,(J^*J-z^2)^{-1}\mathbf e_1\rangle =\langle \delta_0,(J^*J-z^2)^{-1}J^*\delta_0\rangle \\
&=\langle \delta_0,J(J^*J-(z^2)^*)^{-1}\delta_0 \rangle^*=\int\limits_{[0,+\infty)} \frac{x \psi^*(x) d \nu(x)}{x^2-z^2}, 
\end{align*}

For the bottom right entry, we compute
\begin{align*}
M_{1,1}(z^2)&=\langle \mathbf e_1,(J^*J-z^2)^{-1}\mathbf e_1\rangle=\langle \mathbf \delta_0,J(J^*J-z^2)^{-1}J^*\delta_0\rangle\\
&=\langle \delta_0,(JJ^*-z^2)^{-1}JJ^*\delta_0\rangle 
\end{align*}
by Lemma~\ref{lem:BorelJJstar}. Denoting the anti-unitary operator $C:\ell^2(\mathbb N_0)\to \ell^2(\mathbb N_0)$, $C\psi=\psi^*$, we have $CJC=J^*$ as operators, so that since $C\delta_0=\delta_0$,
\begin{align*}
\langle \delta_0,(JJ^*-z^2)^{-1}JJ^*\delta_0\rangle &=\langle C\delta_0,(JJ^*-z^2)^{-1}JJ^*C\delta_0\rangle=\langle \delta_0,C(JJ^*-z^2)^{-1}CCJJ^*C\delta_0\rangle^*\\
&=\langle \delta_0,(J^*J-(z^*)^2)^{-1}J^*J\delta_0\rangle^*=\langle \delta_0,(J^*J-z^2)^{-1}J^*J\delta_0\rangle
\end{align*}
so that by \eqref{eq:nudef}, we have
\begin{align*}
&M_{1,1}(z^2)=\int\limits_{[0,+\infty)} \frac{x^2  d \nu(x)}{x^2-z^2}.
\end{align*}

By the above integral representations for $M_{0,0}$ and $M_{1,1}$, a direct verification using $\nu(\bbR) = 1$ gives \eqref{eq:Mdiagonalentryidentity}. 

We express the relevant entries of $(\mathbf J-z)^{-1}$ in terms of the $M_{i,j}$. By block matrix inversion, we have 
\begin{align*}
(\mathbf J-z)^{-1}\psi= \begin{pmatrix} -\frac1z+\frac1{z^2}J(1/z J^*J-z)^{-1}J^*& \frac1z J(1/zJ^*J-z)^{-1}\\
\frac1z(1/zJ^*J-z)^{-1}J^*& (1/zJ^*J-z)^{-1}
\end{pmatrix}\psi
\end{align*}
for $\psi\in D(\bfJ)$. Thus, with the notation $\delta_0^0:=\delta_0\oplus 0$, $\delta_0^1:=0\oplus \delta_0$, we compute
\begin{align*}
\calR_{0,0}(z) & = \langle\delta_0^0, (\mathbf J-z)^{-1}\delta_0^0 \rangle = -1/z+\langle J^*\delta_0,1/z(J^*J-z^2)^{-1}J^*\delta_0\rangle \\&=\frac1z(-1+\langle \mathbf{e}_1,(J^*J-z^2)^{-1}\mathbf{e}_1\rangle)
=\frac{1}{z}(-1+M_{1,1}(z^2)),
\end{align*}
and similarly, 
\begin{align*}
\calR_{0,1}(z)&=\langle \delta_0^0,(\mathbf J-z)^{-1}\delta_0^1 \rangle=\langle J^*\delta_0,(J^*J-z^2)^{-1}\delta_0\rangle\\
&=\langle \mathbf e_1,(J^*J-z^2)\mathbf e_0 \rangle =M_{0,1}(z^2).
\end{align*}
and 
\begin{align*}
\calR_{1,0}(z)&=\langle \delta_0^1,(\mathbf J-z)^{-1}\delta_0^0 \rangle=\langle \delta_0,1/z(1/zJ^*J-z^2)^{-1}J^*\delta_0\rangle\\
&=\langle\mathbf e_0,1/z(1/zJ^*J-z^2)^{-1}\mathbf e_1\rangle=M_{1,0}(z^2),
\end{align*}
and finally
\begin{align*}
\calR_{1,1}(z)&=\langle \delta_0^1,(\mathbf J-z)^{-1}\delta_0^1 \rangle=\langle \delta_0,(1/zJ^*J-z)^{-1}\delta_0\rangle\\
&=zM_{0,0}(z^2).
\end{align*}
After setting $\frac{1}{z}(-1+M_{1,1}(z^2))=zM_{0,0}(z^2)$, this gives \eqref{eq:RM}. Using a partial fraction decomposition on the integral representations for $M_{i,j}$ gives \eqref{eq:Rstructure}. By Stieltjes inversion, the data $(\nu,\psi)$ uniquely determines $\mathcal R$.

\end{proof}

By the previous lemma, nontagential asymptotics for $M$ and $\cR$ may be related. We note that we have $M(w)=O(1/\sqrt{w})$ nontangentially in $\mathbb C\setminus \mathbb R$, as may be seen either through the relationship to $\calR$ just proven, or through the integral identities for $M_{i,j}$ derived above.

\begin{lemma}\label{lem:RvMsim}
Fix sequences $z_j$ with $\arg(z_j)\in [\epsilon,\pi-\epsilon]$ for an $\epsilon>0$, and $w_j\in \bbC\setminus[0,+\infty)$ with $w_j^{1/2}=z_j$. Let $\bbJ$ and $\tilde\bbJ$ be two block Jacobi matrices whose coefficients have the form \eqref{eq:blockform}, and denote by $\cR, \tilde \cR$ be their Weyl matrix functions, and let $M,\tilde M$ be defined by \eqref{eq:distortedM} corresponding to $V\mathbb J V^*=\mathbf J,V\tilde{\mathbb{J}} V^*=\tilde{\mathbf{J}}$ respectively. Then, 
\begin{align}\label{eq:Msim}
\left \|w_j^{\frac{\sigma_3}{4}}\left(M(w_j)-\tilde M(w_j)\right)w_j^{\frac{\sigma_3}{4}}\right\|=\mathcal{O}(w_j^{-N}), \quad j\to\infty
\end{align}
for an $N\in \bbN$ if and only if 
\begin{align}\label{eq:Rsim}
\|\cR(z_j)-\tilde \cR(z_j)\|=\mathcal{O}(z_j^{-2N}), \quad j\to\infty.
\end{align}
\end{lemma}

\begin{proof}
By \eqref{eq:RM} and the relation \eqref{eq:Mdiagonalentryidentity}, 
\begin{align*}
\|\cR(z_j)-\tilde \cR(z_j)\|&=\left\|\begin{pmatrix}
z_j(M_{0,0}(z_j^2)-\tilde M_{0,0}(z_j^2))&M_{0,1}(z_j^2)-\tilde M_{0,1}(z_j^2)\\
M_{1,0}(z_j^2)-\tilde M_{1,0}(z_j^2)&\frac{1}{z_j}(M_{1,1}(z_j^2)-\tilde M_{1,1}(z_j^2))
\end{pmatrix}  \right\|\\
&=\left\|\begin{pmatrix}
w_j^{1/2}(M_{0,0}(w_j)-\tilde M_{0,0}(w_j))&M_{0,1}(w_j)-\tilde M_{0,1}(w_j)\\
M_{1,0}(w_j)-\tilde M_{1,0}(w_j)&\frac{1}{w_j^{1/2}}(M_{1,1}(w_j)-\tilde M_{1,1}(w_j))
\end{pmatrix}  \right\|\\
&=\left \|w_j^{\frac{\sigma_3}{4}}\left(M(w_j)-\tilde M(w_j)\right)w_j^{\frac{\sigma_3}{4}}\right\|.
\end{align*}
Thus, \eqref{eq:Msim} holds if and only if \eqref{eq:Rsim} is satisfied.  
\end{proof}

The following lemma allows for the computation of $A_n$ and $B_n$ through inductive coefficient stripping, and is the final lemma needed before our proof of Theorem~\ref{thm:LocalBorgMarchenko}.

\begin{lemma}\label{lem:inductivecrank}
Let $\bbJ$ and $\tilde\bbJ$ be two block Jacobi matrices whose coefficients have the form \eqref{eq:blockform}, and let $\cR, \tilde \cR$ be their Weyl matrix functions. Fix a sequence $z_j\to\infty$ as $j\to\infty$ nontangentially in $\bbC_+$. Then, for $k \ge 2$ we have
\begin{align}\label{eq:0107simbase}
\lVert \cR(z_j) - \tilde\cR(z_j) \rVert = \mathcal{O}(z_j^{-k-2}),\quad j\to\infty 
\end{align} 
if and only if $A_0 = \tilde A_0$, $B_0 = \tilde B_0$, and 
\begin{align}\label{eq:0107strippedbase}
\lVert \cR_1(z_j) - \tilde\cR_1(z_j) \rVert = \mathcal{O}(z_j^{-k}),\quad j\to\infty.
\end{align}

\end{lemma}

\begin{proof}
Suppose \eqref{eq:0107simbase} holds for a $k\geq 2$. Then, we have
\begin{align*}
\lim_{j\to\infty}z_j^\ell (\calR(z_j)- \tilde \calR (z_j))=0, 
\end{align*}
for $\ell \leq 3$. Thus, by \eqref{eq:Rexpansion} of Lemma~\ref{lem:blocks}, $\tilde B_0= B_0$, and
\begin{align*}
A_0=\tilde A_0U,
\end{align*}
for $U$ a unitary matrix. Since $A_0,\tilde A_0$ have the form \eqref{eq:blockform} for $a_0,\tilde a_0>0$, this implies $U=I$ and $A_0=\tilde A_0$, after multiplying by $\tilde A_0^{-1}$. Thus, since 
\begin{align*}
\|\calR_1(z)-\tilde\calR_1(z)\|&\leq \| A_0\|\|A_0^{-1}\|\| \calR(z)^{-1}-\tilde\calR(z)^{-1}\|\\
&\leq \| A_0\|\|A_0^{-1}\|\| \calR(z)^{-1}\|\|\tilde\calR(z)^{-1}\| \|\calR(z)-\tilde\calR(z)\|
\end{align*}
by the second resolvent identity, and
\begin{align*}
\calR(z_j)^{-1}, \tilde\calR(z_j)^{-1}=-z_j \mathbb{I} +\mathcal O(1),\quad j\to \infty,
\end{align*}
by \eqref{eq:Rexpansion}, we may conclude \eqref{eq:0107strippedbase}.

Similarly, for the converse, we have after rearranging \eqref{eq:oncestripped} and using the assumption that $A_0 = \tilde A_0$, $B_0 = \tilde B_0$,
\begin{align*}
\|\calR(z_j)-\tilde\calR(z_j)\|&= \|(A_0\calR_1(z_j)A_0^*+B_0-z_j)^{-1}-(A_0\tilde \calR_1(z_j)A_0^*+B_0-z_j)^{-1}\|\\
&=\mathcal{O}(z_j^{-2})\|\calR_1(z_j)-\tilde \calR_1(z_j)\|
\end{align*}
by the second resolvent identity, and since 
\begin{align*}
(A_0\calR_1(z_j)A_0^*+B_0-z_j)^{-1},(A_0\tilde \calR_1(z_j)A_0^*+B_0-z_j)^{-1}=\mathcal{O}(1/z_j), j\to\infty.
\end{align*}
Thus, by \eqref{eq:0107strippedbase}, we have that \eqref{eq:0107simbase} holds along the same sequence.

\end{proof}

We may now prove Theorem~\ref{thm:LocalBorgMarchenko}. Both statements make use of the relation given by Proposition~\ref{prop:RMrelationship}, and induction using Lemma~\ref{lem:inductivecrank}.

\begin{proof}[Proof of Theorem~\ref{thm:LocalBorgMarchenko}]
Let $J,\tilde J$ be proper and denote the corresponding functions by $M, \tilde M$, and denote by $\calR$ and $ \tilde \calR$ the Weyl-matrices for 
\begin{align}\label{eq:1230}
V^*\bfJ V=\bbJ, \quad V^*\tilde \bfJ V= \tilde \bbJ,
\end{align}
respectively. We will repeatedly use that by Lemma~\ref{lem:RvMsim}, the equation \eqref{eq:asymptotic} is equivalent to
\begin{align}\label{eq:basecase}
\|\calR(z_j)-\tilde \calR(z_j)\|=\mathcal{O}(z_j^{-2N})
\end{align}
with $z_j=w_j^{1/2}$ and $\arg(z_j)\in [\epsilon,\pi-\epsilon]$ for some $\epsilon>0$. 

\eqref{item:fixedsim}$\implies$\eqref{item:coeffeq}. We proceed by induction on $N$. For the base case $N=2$, by the equivalence of \eqref{eq:asymptotic} and \eqref{eq:basecase}, and Lemma~\ref{lem:inductivecrank}, we have
\begin{equation}\label{eq:ab0}
A_0=\tilde A_0, B_0=\tilde B_0
\end{equation}
so that $a_0=\tilde a_0$, $b_0=\tilde b_0$.
For the inductive step, we assume the claim holds at $N\geq 2$, in other words \eqref{eq:asymptotic} holding along a fixed sequence with suitable argument implies equality of the first $N-1$ parameters. Now suppose
\begin{align}
\left\| w_j^{\frac{\sigma_3}{4}}  \left( M(w_j)- \tilde M(w_j)  \right)w_j^{\frac{\sigma_3}{4}} \right\|=\mathcal{O}(w_j^{-N-1}),\; j\to \infty.
\end{align}
along a fixed sequence with $w_j\to \infty$ and $\arg(w_j)\in [\epsilon,2\pi-\epsilon]$ for some $\epsilon>0$. Then, by Lemma~\ref{lem:RvMsim}, 
\begin{align*}
\|\calR(z_j)-\tilde \calR(z_j)\|=\mathcal{O}(z_j^{-2N-2})
\end{align*}
for $z_j\in \mathbb C_+$, $z_j^2=w_j$. Thus, by Lemma~\ref{lem:inductivecrank}, we have \eqref{eq:ab0}, so that $a_0=\tilde a_0$ and $b_0=\tilde b_0$, and 
\begin{align*}
\|\calR_1(z_j)-\tilde \calR_1(z_j)\|=\mathcal{O}(z_j^{-2N}).
\end{align*}
Now, using the equivalence given by Lemma~\ref{lem:RvMsim}, we have \eqref{eq:asymptotic} for $\mathbf J$ and $\tilde {\mathbf{J}}$ formed through the once stripped matrices $ J_1=S_+^* JS_+$ and $\tilde {J}_1=S_+^*\tilde{ J}S_+$, so that by the inductive hypothesis, $a_n=\tilde a_n$ and $b_n=\tilde b_n$ for $1\leq n\leq N-1$, completing the proof.


\eqref{item:coeffeq}$\implies$\eqref{item:simall} is also proved by induction on $N$. For $N=2$, we note that $a_0=\tilde a_0$ and $b_0=\tilde b_0$ implies \eqref{eq:ab0}, so that by \eqref{eq:Rexpansion} 
\begin{align*}
\|\calR(z)-\tilde{\calR}(z)\|=\mathcal O(z^{-4}),
\end{align*}
and the equivalence of Lemma~\ref{lem:RvMsim} proves the base case. We now assume desired implication holds at $N$: equality of the first $N-1$ coefficients in a proper Jacobi matrix implies \eqref{eq:asymptotic1} for the corresponding $M$. 

Let $J$ and $\tilde J$ be proper and suppose $a_n=\tilde{a}_n$, $b_n=\tilde{b}_n$ for $0\leq n\leq N-1$. We denote by $J_1$ and $\tilde J_1$ the once stripped Jacobi matrices
\begin{align*}
J_1=S_+^*JS_+,\;\tilde J_1=S_+^*\tilde{J}S_+
\end{align*}
and we define 
\begin{align*}
\mathbf J_1=\begin{pmatrix}
0&J_1\\J_1^*&0
\end{pmatrix},\;\tilde {\mathbf J_1}=\begin{pmatrix}
0&\tilde J_1\\ \tilde J_1^*&0
\end{pmatrix}
\end{align*} 

Then, applying the inductive hypothesis to $\tilde{\mathbf{J}}_1$ and $\mathbf J_1$, along with the equivalence given by Lemma~\ref{lem:RvMsim} yields
\begin{align}\label{eq:step121b}
\|\calR_1(z)-\tilde \calR_1(z)\|=\mathcal{O}(z^{-2N}),
\end{align}
for $z\to\infty$ non tangentially in $\mathbb C_+$. Together with \eqref{eq:ab0} and Lemma~\ref{lem:inductivecrank}, we have
\begin{align}\label{eq:basesim}
\|\cR(z)-\tilde \cR(z)\|=\mathcal{O}(z^{-2N-2})
\end{align}
as $z\to \infty$ nontangentially in $\bbC_+$. So, by Lemma~\ref{lem:RvMsim} again, we see that 
\begin{align*}
\left\|w^{\frac{\sigma_3}{4}}\left(M(w)-\tilde M(w)\right)w^{\frac{\sigma_3}{4}}\right\|=\mathcal{O}(w^{-N-1}), \quad w \to\infty
\end{align*}
nontangentially in $\bbC\setminus[0,+\infty)$, as required.

\eqref{item:simall}$\implies$\eqref{item:fixedsim} is trivial.
\end{proof}

With this theorem in hand, the injectivity part of the proof of Theorem~\ref{thm:inverse} is quick. Using the work of \cite{PushnitskiStampach}, non-degeneracy of the measure $\mu$ follows. The proof of surjectivity proceeds in two steps: first, given a completely determinate measure, we we use right orthogonal polynomials to construct a block Jacobi matrix, which by our hypothesis, corresponds to a unique operator that is the unique self-adjoint extension guaranteed by complete determinacy. Since this operator is only defined up to equivalence class, where the equivalence is given by conjugation by a block diagonal unitary matrix, the second step ensures this equivalence class contains a matrix of the correct form. 


\begin{proof}[Proof of Theorem~\ref{thm:inverse}] 
If $J$ is a proper Jacobi matrix, $\nu$ has infinite support and finite moments by Lemma~\ref{lem:momentsfinite}. By \cite[Lemma 2.1]{DamPushnSim08}, the definition of non-degeneracy \eqref{eq:nondeg} is equivalent to the definition used in \cite{PushnitskiStampach}. Thus, by \cite[Lemma 6.1]{PushnitskiStampach}, $\mu$ is non-degenerate. Finally, $\mu$  is the spectral measure for a block Jacobi matrix with deficiency indices $(0,0)$, so that $\mu$ is completely determinate.

Let $J,\tilde J$ be proper and assume they correspond to the same matrix measure. 
Then, $M(z)=\tilde M(z)$, so  \eqref{eq:asymptotic} holds for all $N$ and $J$ and $\tilde J$ have the same Jacobi parameters, so that $J =\tilde  J$ by Theorem~\ref{thm:LocalBorgMarchenko}.  

It remains to prove surjectivity, i.e., that every $2\times 2$ completely determinate matrix measure with the right symmetries corresponds to a proper Jacobi matrix $J$. Let $(\nu,\psi)$ be such that the corresponding $\mu$ is completely determinate. Using non-degeneracy, we may consider a sequence of right orthonormal matrix polynomials $p_n$ with respect to the measure 
\begin{align*}
d\mu(x):=\begin{pmatrix}
1&\psi_o(x)\\\psi_{o}(x)^*&1
\end{pmatrix}d\nu_{e}(x).
\end{align*}
Using the recurrence relation satisfied by the $p_n$, 
\begin{align}\label{eq:pnrecursion}
xp_n(x)=p_{n-1}(x)A_{n-1}^*+p_{n}(x)B_n+p_{n+1}(x)A_n, n\in \bbN_0
\end{align}
with the convention that $p_{-1}(x)=0$, we may assemble a block Jacobi matrix $\bbJ_0(A,B)$ defined on $\ell^2_c(\bbN_0)$ corresponding these coefficients. We note that $\mathbb J$ is only uniquely determined up to equivalence class, where the equivalence is given by conjugation by a block diagonal unitary matrix. By complete determinacy, $\bbJ_0(A,B)$ has deficiency indices $(0,0$ and is thus essentially self-adjoint, and we denote its self-adjoint extension by $\bbJ=\ol{\bbJ_0(A,B)}$.

We first claim that $d\mu$ is the spectral measure for $\bbJ$:  Let $d\tilde\mu$ be the spectral measure for $\bbJ$: 
\begin{align*}
\langle \delta_{j},(\bbJ-z)^{-1}\delta_k\rangle =\int\limits_{\bbR}\frac{1}{x-z}d\tilde \mu_{j,k}(x)
\end{align*}
for $ 0\leq j,k\leq 1$ and $\;z\in\bbC\setminus \bbR$. Then, by the spectral theorem, there is a unitary map

\begin{align*}
\mathcal{U}: \ell^2(\bbN_0)\to L^2(\bbR^2;\bbC^{2}, d \tilde \mu),
\end{align*}
with $\mathcal{U}(\delta_j)=e_j,\;0\leq j\leq 1$ and
\begin{align*}
(\calU\mathbb J \psi)(x)=x(\calU\psi)(x)
\end{align*}
for any $\psi\in D(\bbJ)$. In particular, since 
\begin{align*}
\bbJ\ell^2_c(\bbN_0)\subseteq \ell^2_c(\bbN_0)\subseteq D(\bbJ ),
\end{align*}
we have
\begin{align}\label{eq:powers1}
(\calU\mathbb J^n \psi)(x)=x^n(\calU\psi)(x)
\end{align} 
for any $\psi\in \ell^2_c(\bbN_0)$ and $n\in \bbN_0$.

In what follows, it will be convenient to identify $\ell^2(\bbN_0)$ and $\ell^2(\bbN_0;\bbC^2)$ with the basis
\begin{align*}
\hat\delta_{j,k}=\delta_j e_{k},\;j\in \bbN_0, k\in \{0,1\}.
\end{align*}
We may construct a sequence of orthonormal polynomials by examining $\calU(\hat{\delta}_{n,k})$ for $k\in \{0,1\}$. We note that by the expression for the $\bbJ$:
\begin{align*}
(\bbJ f)_n=A_{n-1}^*f_{n-1}+B_nf_n+A_{n+1}f_{n+1},\;f\in \ell^2(\bbN_0;\bbC^2), A_{-1}:=0,
\end{align*}
and since the $A_n$ are invertible, we have
\begin{align*}
\linspan\{\hat\delta_{j,k}: 0\leq j\leq n,0\leq k\leq 1 \}=\linspan\{ \bbJ^{j}\hat{\delta}_{0,k}:0\leq k\leq 1,0\leq j\leq n\},
\end{align*}
so that equation \eqref{eq:powers1} implies $\calU(\hat{\delta}_{n,k})$ are polynomials of degree at most $n$. Furthermore, by unitarity of the $\calU$, we have
\begin{align}\label{eq:orthonormality1}
\langle \calU(\hat\delta_{n,k}), \calU(\hat\delta_{m,j})\rangle=\begin{cases} 1,&n=m,j=k\\
0,&\text{otherwise}
\end{cases}
\end{align}
so that the matrix polynomial $q_{n}$ made up of the columns $\calU(\hat\delta_{n,0}),\calU(\hat\delta_{n,1})$,
\begin{align*}
q_{n}(x)_{0\leq j,k\leq 1}:=e_{j}^{*}\calU(\hat\delta_{n,k})
\end{align*}
are orthonormal. We relate them to the entries of $\bbJ_{nm}$ with respect to the $\hat\delta$ basis:
\begin{align*}
(\bbJ_{nm})_{0\leq j,k\leq 1}&=\langle \hat{\delta}_{n,j},\bbJ \hat{\delta}_{m,k}\rangle
=\langle q_{n}(x)e_j,xq_{m}(x)e_k\rangle,
\end{align*}
so that 
\begin{align*}
\bbJ_{n,m}=\int\limits_{\bbR} q_{n}(x)^{*}d\tilde \mu(x)xq_{m}(x).
\end{align*}
Thus, the $q_n$ satisfy the same recursive equation \eqref{eq:pnrecursion} as the $p_n$, and since  $q_0=p_0=\mathbb I$, we have $q_n\equiv p_n$.

Since the $\{p_k:0\leq k\leq n\}$ form a basis for the polynomials of degree at most $n$, we may write for each $n\in \bbN_0$,
\begin{align*}
x^n\mathbb I=\sum_{j=0}^nC_{j,n}p_j(x)
\end{align*}
with a unique choice of $C_{j,n}$. Then, 
\begin{align*}
\int\limits_{\bbR}x^nd\tilde \mu(x)=\int\limits_{\bbR}x^nd\mu(x)=C_{0,n}.
\end{align*}
Thus, by determinacy of the measure $d\mu$, we have $d\mu=d\tilde \mu$.

We now show that the representative $\mathbb J$ may be chosen to have blocks of the form \eqref{eq:blockform} by inductively selecting blocks
\begin{align*}
W_j\in \mathbb C^{2\times 2}, W_jW_j^*=\mathbb I,
\end{align*}
in the block diagonal unitary matrix 
\begin{align*}
W=\begin{pmatrix}
W_0&0&0&\dots\\
0&W_1&0&\dots\\
0&0&W_2&\dots\\
\vdots&\vdots&\vdots&\ddots
\end{pmatrix},
\end{align*}
so that 
\begin{align*}
W^*\mathbb J(A_n,B_n)W=\mathbb J(A_n',B_n')=:\bbJ'
\end{align*}
is a matrix of the correct form.

For the base case, we note that by Lemma~\ref{lem:blocks}, the entries $A_0$ and $B_0$ have the form 
\begin{align*}
A_0=\begin{pmatrix}
0&a_0\\
a_0&0
\end{pmatrix}U_1, \;B_0=\begin{pmatrix}
0&b_0\\
b_0^*&0
\end{pmatrix}, a_0>0,
\end{align*}
for $U_1$ a unitary matrix in $\mathbb C^{2\times 2}$, so that selecting $W_0=\mathbb I$ and $W_1=U_1^*$ ensures $A_0'$ and $B_0'$ are of the correct form \eqref{eq:blockform}. 

We note that $W_0,\dots, W_n$ only effect the first $n$ block coefficients: $A_j,B_j$ for $0\leq j\leq n$. We suppose now that $W_0,\dots,W_{n+1}$ are selected so that $A_j',B_j'$ are of the form \eqref{eq:blockform} for $0\leq j\leq n$. We show that we may select $W_{n+2}$ so that $A_{n+1}'$ and $B_{n+1}'$ are of the correct form. By Proposition~\ref{prop:stripping} and the inductive hypothesis, the Weyl-matrix $\calR_n$ corresponding to the $n$-times stripped matrix $\mathbb J_n':=(S_+^*)^{2n}\mathbb J'S_+^{2n}$ is also symmetric. Applying Lemma~\ref{lem:blocks} to $\mathbb J_n'$, we see 
\begin{align*}
A_{n+1}'=\begin{pmatrix}
0&a_{n+1}'\\
a_{n+1}'&0
\end{pmatrix}U_{n+2}, \;B_{n+1}'=\begin{pmatrix}
0&b_{n+1}'\\
(b_{n+1}')^*&0
\end{pmatrix}, a_{n+1}'>0,
\end{align*}
for $U_{n+2}$ a unitary matrix in $\mathbb C^{2\times 2}$, so that we may select $W_{n+2}=U_{n+2}^*$. So, by induction, we may select $W$ so that the equivalence class of $\mathbb J$ contains a matrix of the correct form.


Then, denoting the above sequences $a=(a_n')_{n\in\bbN_0}, b=(b_n')_{\bbN_0}$, we have 
\begin{align*}
V\bbJ' V^*=\bfJ=\begin{pmatrix}0&J(a,b)\\J(a,b^*)&0 \end{pmatrix},
\end{align*}
so that $(\nu,\psi)$ is the spectral data of $J(a,b)$ by Proposition~\ref{prop:RMrelationship}.
\end{proof}

For the proof of Theorem~\ref{thm:continuity}, we use a well-known result allowing for the conclusion of strong resolvent convergence for self-adjoint operators from convergence on a common core, along with a characterization of pointwise convergence of Herglotz functions in terms of the weak convergence of their measures. On the inverse side, the conditions \eqref{eq:weak} give pointwise convergence of the relevant Herglotz functions, which only corresponds to convergence of certain matrix elements of the relevant resolvent. This falls well short of strong resolvent convergence, making statements about convergence of the coefficients $a_n,b_n$ out of reach.

\begin{proof}[Proof of Theorem~\ref{thm:continuity}]
Suppose $J^N:=J(a^{N},b^{N})$ and $J:=J(a^\infty,b^\infty)$ are proper and $J^N\to J$ in the sense of \eqref{eq:strong}. As before, the operators 
\begin{align*}
\bfJ^N=\begin{pmatrix}
0&J^N\\
(J^N)^*&0
\end{pmatrix},  \quad 
\bfJ=\begin{pmatrix}
0&J\\
J^* &0
\end{pmatrix}
\end{align*}
are self-adjoint with domains given by \eqref{eq:domain}, with a common core $\ell^2_c(\bbN_0)\oplus \ell^2_c(\bbN_0)$. Then, by the assumption \eqref{eq:strong}, for $\psi\in \ell^2_c(\bbN_0)\oplus \ell^2_c(\bbN_0)$,
\begin{align*}
\bfJ^N\psi\to \bfJ\psi,
\end{align*} 
so that $\bfJ^N\to \bfJ$ in strong-resolvent sense. Thus, we have convergence of the Herglotz functions $\calR,\calR^N$ corresponding to $\bbJ=V^*\bfJ V,\bbJ^N=V^*\bfJ^N V$:
\begin{align*}
\calR^N(z)\to \calR(z),\quad z\in\bbC\setminus \bbR
\end{align*}
where by \eqref{eq:Rstructure},
\begin{align}\label{eq:1021one}
\mathcal R^N(z)=\int\limits_{\bbR}\begin{pmatrix}
1&\psi_{o}^N(x)\\\psi_{o}^N(x)^*& 1
\end{pmatrix}\frac{d\nu_e^N(x)}{x-z}=:\int\limits_{\bbR}W_N(x)\frac{d\nu_e^N(x)}{x-z}
\end{align}
and
\begin{align}\label{eq:1021two}
\mathcal R(z)=\int\limits_{\bbR}\begin{pmatrix}
1&\psi_{o}(x)\\\psi_{o}(x)^*& 1
\end{pmatrix}\frac{d\nu_e(x)}{x-z}=:\int\limits_{\bbR}W(x)\frac{d\nu_e(x)}{x-z}.
\end{align}

This implies convergence of the scalar Herglotz functions
\begin{align*}
v^{\,*}\calR^N(z)v\to v^{\,*}\calR(z)v,
\end{align*}
for any $v\in \bbC^2$ and $z\in \bbC\setminus \bbR$.  Thus, by \cite[Proposition 7.28]{LukicBook} we have the weak convergence,
\begin{align}\label{eq:weakconvergence}
&\lim_{N\to\infty}\int\limits_\bbR h(x)v^{\,*}W_N(x)v d\nu^N(x)= \int\limits_{\bbR}h(x)v^{\,*}W(x)v d\nu(x)
\end{align} 
for any $h\in C_0(\bbR)$. Specializing to $v =e_1$ in \eqref{eq:weakconvergence} and using \eqref{eq:1021one}, \eqref{eq:1021two} then gives
\begin{align*}
&\lim_{N\to\infty}\int\limits_\bbR h(x)d\nu^N(x)= \int\limits_{\bbR}h(x)d\nu(x).
\end{align*} 

Meanwhile, applying the polarization identity to the sesquilinear forms $u^{\,*} W_Nv$, $u^{\,*} Wv $ gives the convergence 
\begin{align*}
&\lim_{N\to\infty}\int\limits_\bbR h(x)e_1^{\,*}W_N(x)e_2d\nu^N(x)= \int\limits_{\bbR}h(x)e_1^{\,*}W(x)e_2d\nu(x)\\
&\iff \lim_{N\to\infty}\int\limits_\bbR h(x)\psi_o^N(x)d\nu^N(x)= \int\limits_{\bbR}h(x)\psi_o(x)d\nu(x)
\end{align*}
for any $h\in C_0(\bbR)$, establishing convergence in the sense of \eqref{eq:weak}.
\end{proof}

Finally, we prove Theorem~\ref{thm:specialclasses}. The proofs of \eqref{itm:self-adjointness} and \eqref{itm:freecriterion} are similar, each relying on \eqref{eq:phasedef} and Stieltjes inversion for the forward direction, and an inductive argument using coefficient stripping for the converse. 

\begin{proof}[Proof of Theorem~\ref{thm:specialclasses}]

\eqref{itm:self-adjointness}: Suppose $J=J^*$. Then, $f(|J|)$ and $J$ commute for $f\in \calB(\bbR)$, so that for $z\in \bbC\setminus \bbR$, 
\begin{align*}
\langle \delta_0,J R_{|J|}(z)\delta_0\rangle^*=\langle \delta_0,JR_{|J|}(z^*)\delta_0\rangle  
\end{align*}
so that by \eqref{eq:phasedef}, we have
\begin{align*}
\left(\int\limits_{[0,\infty)}\frac{s\psi(s)}{s-z}d\nu(s) \right)^*=\int\limits_{[0,\infty)}\frac{s\psi(s)}{s-z^*}d\nu(s).
\end{align*}
Then, by \eqref{eq:reality} of Lemma~\ref{lem:Cstieltjes}, we have 
\begin{align*}
s\psi(s)^*d\nu(s)=s\psi(s)d\nu(s)
\end{align*}
and $\psi(s)=\psi(s)^*$ for $\nu$ almost every $s>0$. Finally, $\psi(0)=0$ by our normalization, yielding the claim.

Suppose now $\Im(\psi(s))=0$ for $\nu$-a.e. $s\geq 0$. It suffices to show $b_n\in \bbR$ for all $n\geq 0$. We proceed by induction on $n$.

Let $\bbJ=V^*\bfJ V$ for $\bfJ$ defined as in \eqref{eq:domain} and let $\calR$ be the associated Weyl-matrix.  The base case follows immediately from the formula for $b_0$ given by \eqref{eq:b0a0moments} of Lemma~\ref{lem:blocks} applied to $\bbJ$. Suppose now $b_j\in \bbR$ for $0\leq j\leq n$. For $j\geq 1$, denote by $\bfJ_{j}$ the matrix 
\begin{align*}
\bfJ_{j}=\begin{pmatrix}
0&J_j\\ J_j^*&0
\end{pmatrix}
\end{align*}

for $J_j=(S_+^*)^j JS_{+}^j$, and denote by 
\begin{align*}
\bbJ_{j}=(S_+^*)^{2j}\bbJ S_+^{2j}=V^*\bfJ_{j}V.
\end{align*}
Finally, denote by $\calR_{j}$ the Weyl matrix corresponding to $\bbJ_{j}$. By Proposition~\ref{prop:RMrelationship}, we may write 
\begin{align}\label{eq:strippedRj}
\mathcal R_{n+1}(z)=\int\limits_{\bbR}\begin{pmatrix}
1&\psi_{o}^{n+1}(x)\\\psi_{o}^{n+1}(x)^*& 1
\end{pmatrix}\frac{d\nu_e^{n+1}(x)}{x-z}
\end{align}
for $(\nu^{n+1},\psi^{n+1})$ corresponding to $(S_+^*)^{n+1} JS_{+}^{n+1}$. Since $b_n\in \bbR$ for $0\leq j\leq n$, by induction and the formula \eqref{eq:strippedformula}, the entries of $\calR_{n+1}$ satisfy
\begin{align*}
(\calR_{n+1}(z))_{i,j}^*=(\calR_{n+1}(z^*))_{i,j}.
\end{align*}
In particular, applying \eqref{eq:reality} to $(\calR_{n+1}(z))_{0,1}$ and $(\calR_{n+1}(z))_{1,0}$, we see $\Im (\psi^{n+1}_{o}(s))=0$ for $\nu_e^{n+1}$ almost every $s\in \bbR$. Applying Lemma~\ref{lem:blocks}, \eqref{eq:b0a0moments} to $\bbJ_{n+1}$ then gives $b_{n+1}\in \bbR$.

\eqref{itm:freecriterion}: Suppose $b_n=0$ for all $n\geq 0$. Let $\Omega$ be the unitary diagonal operator with coefficients 
\begin{align*}
\langle \delta_j,\Omega \delta_k\rangle=(-1)^k\delta_{j,k},\quad j,k\geq 0.
\end{align*}

Since $b_n\equiv 0$, we have $\Omega J\Omega=-J$, as well as $J=J^*$. Thus, for $z\in \bbC\setminus \bbR$, 
\begin{align*}
\Omega (J^*J-z)\Omega=\Omega J^2\Omega-z=J^2-z,
\end{align*}
so that $\Omega R_{J^*J}(z)\Omega=R_{J^*J}(z)$. Then, since $\Omega$ is unitary and satisfies $\Omega\delta_0=\delta_0$, we have 
\begin{align*}
\langle \delta_0,J R_{J^*J}(z)\delta_0\rangle=\langle \Omega \delta_0,\Omega JR_{J^*J}(z)\Omega\delta_0\rangle  =-\langle \delta_0,JR_{J^*J}(z)\delta_0\rangle 
\end{align*}
so that by \eqref{eq:phasedef}, 
\begin{align*}
\int\limits_{[0,+\infty)} \frac{s\psi(s)}{s^2-z}d\nu(s)=0,\quad \forall z\in \bbC\setminus \bbR.
\end{align*}
Making the change of variables $\phi(s)= \sqrt{s}$, we have 
\begin{align*}
\int\limits_{[0,+\infty)} \frac{s\psi(s)}{s^2-z}d\nu(s)=\int\limits_{[0,+\infty)} \frac{\sqrt{s}\psi(\sqrt{s})}{s-z}d\phi^{-1}_*\nu(s)=0 
\end{align*}
for all $z\in \bbC\setminus \bbR$. By Stieltjes inversion, we have $\sqrt{s}\psi(\sqrt{s})d\phi_*^{-1}\nu=0$, so that 
\begin{align*}
0=\phi^{-1}_*\nu(\{s>0:\psi(\sqrt{s})\ne 0\})=\nu(\{t>0:\psi(t)=0\}
\end{align*}
and $\psi(s)= 0$ for $\nu$-a.e. $s>0$. Since $\psi(0)=0$ by normalization, the claim is proved.

The converse follows similarly to the backwards direction of \eqref{itm:self-adjointness} immediately above, and we borrow notation from there. Suppose $\psi(s)=0$ for a.e. $s\geq 0$. We have $b_0=0$ by \eqref{eq:b0a0moments} of Lemma~\ref{lem:blocks} applied to $\bbJ$. Suppose now $b_j=0$ for $0\leq j\leq n$. Since $\psi(s)=0$ for $\nu$-a.e. $s\geq 0$, we see by Proposition \ref{prop:RMrelationship},
\begin{align*}
\calR_{0,1}(z)=\calR_{1,0}(z)=0,\quad z\in \bbC\setminus \bbR.
\end{align*}
Since $b_n=0$ for $0\leq j\leq n$, by induction and the formula \eqref{eq:strippedformula}, the entries of $\calR_{n+1}$ satisfy
\begin{align*}
(\calR_{n+1}(z))_{0,1}=(\calR_{n+1}(z))_{1,0}=0.
\end{align*}
Stieltjes inversion and \eqref{eq:strippedRj} then yields $\psi_o^{n+1}(s)=0$ for $\nu_e^{n+1}$-a.e. $s>0$. Applying Lemma~\ref{lem:blocks}, \eqref{eq:b0a0moments} to $\bbJ_{n+1}$ then gives $b_{n+1}=0$, and by induction, the claim is proven.

\end{proof}

\bibliographystyle{amsplain}
\bibliography{lit}

\end{document}